\documentclass{article}  
\usepackage[pdftex,
bookmarksnumbered,
bookmarksopen,
colorlinks,
citecolor=blue,
linkcolor=blue,]{hyperref}
\usepackage{amsmath}
\usepackage{amssymb}
\usepackage{mathrsfs}
\usepackage{amsfonts}
\usepackage{amsthm}
\usepackage{algorithm}
\usepackage{booktabs}
\usepackage{listings}
\usepackage{boxedminipage}
\usepackage{algorithmic}
\usepackage{stmaryrd}
\usepackage{cite}
\usepackage{relsize}
\usepackage{lscape}
\usepackage[top=1.3in, bottom=1.5in, left=1.5in, right=1.5in]{geometry}
\usepackage{stmaryrd} 
\usepackage{relsize}%
\usepackage{url}
\usepackage{color,xcolor}

\newtheorem{theorem}{Theorem}[section]
\newtheorem{definition}{Definition}[section]
\newtheorem{proposition}{Proposition}[section]

\newtheorem{lemma}{Lemma}[section]
\newtheorem{corollary}{Corollary}[section]
\newtheorem{remark}{Remark}[section]

    \newcommand{\Prob}[1]{\ensuremath{ {\rm Prob}\left\{  #1 \right\}  }}
    \newcommand{\T}{\ensuremath{ \mathbb R^{n_1\times \cdots\times n_d}   }}
   \newcommand{\SSS}{\ensuremath{ \mathbb S^{n^d}   }} 
   \newcommand{\M}{\ensuremath{ \mathbb S_M^{m\times n\times m\times n}   }}  
  \newcommand{\C}{\ensuremath{ \mathbb S_C^{n\times n\times n}   }}   
    
        \newcommand{\E}{\ensuremath{ \mathbb E   }}

    \newcommand{\bigxiaokuohao}[1]{\ensuremath{ \left(  #1 \right) }}      
    \newcommand{\bigjueduizhi}[1]{\ensuremath{ \left|  #1 \right| }}

                \newcommand{\bignorm}[1]{\ensuremath{ \left\|   #1 \right\|  }}

                                \newcommand{\innerprod}[2]{\ensuremath{ \left\langle   #1 , #2\right\rangle }}

	\definecolor{darkgray}{rgb}{0.66, 0.66, 0.66}

\title{On the Largest Singular Value/Eigenvalue of a  Random Tensor 
}

\author{Yuning Yang  \thanks{College of Mathematics and Information Science, Guangxi University, Nanning, 530004, China  (yyang@gxu.edu.cn).}                             
}

\begin{document} 
\maketitle

\begin{abstract}

  This short note presents upper bounds of the expectations of the largest singular values/eigenvalues  of various types of  random tensors  in the non-asymptotic sense. For a  standard Gaussian tensor of size $n_1\times\cdots\times n_d$, it is shown that the expectation of its largest   singular value is upper bounded by $\sqrt {n_1}+\cdots+\sqrt {n_d}$.
    For  the expectation of the largest $\ell^d$-singular value, it is upper bounded by $2^{\frac{d-1}{2}}\prod_{j=1}^{d}n_j^{\frac{d-2}{2d}}\sum^d_{j=1}n_j^{\frac{1}{2}} $. 
    We also derive the upper bounds of the expectations of the   largest Z-/H-($\ell^d$)/M-/C-eigenvalues of symmetric, partially symmetric, and piezoelectric-type Gaussian tensors, which are respectively upper bounded by $d\sqrt n$, $d\cdot 2^{\frac{d-1}{2}}n^{\frac{d-1}{2}}$, $2\sqrt m+2\sqrt n$, and $3\sqrt n$.  
	
\noindent {\bf Keywords:} random tensor;    eigenvalue;  singular value; Gordon's theorem;      Gaussian process    
\end{abstract}

\noindent {\bf  AMS subject classifications.}   15A18, 15A69, 11M50
\hspace{2mm}\vspace{3mm}

\section{Introduction}

 In random matrix theory, the Gordon's theorem says that the expectation of the largest singular value of a Gaussian matrix of size $m\times n$ can be upper bounded by $\sqrt m + \sqrt n$. The proof is based on  the Gaussian process; see  \cite[Theorem 5.32]{vershynin2012introduction}. Using a similar approach, 
this short note directly extends this result to various types singular values and eigenvalues  of  random  higher-order tensors. For a standard Guassian tensor of size $n_1\times \cdots\times n_d$, the expectation of its largest $\ell^2$-singular value can be upper bounded by $\sum^d_{j=1}\sqrt {n_j}$; the expectation of its largest $\ell^d$-singular value can be upper bounded by $2^{\frac{d-1}{2} }\prod^d_{j=1}n_j^{\frac{d-2}{2d}}\sum^d_{j=1}n_j^{\frac{1}{2}}$. We then consider symmetric, partially symmetric, and piezoelectric-type Gaussian tensors. For a $d$-th order $n$-dimensional symmetric Gaussian tensor, the largest Z- ($\ell^2$-) eigenvalue can be upper bounded by $d\sqrt n$; its largest H- ($\ell^d$-) eigenvalue can be upper bounded by $d\cdot 2^{\frac{d-1}{2}}n^{\frac{d-1}{2}}$. For a fourth-order partially symmetric Gaussian tensor of size $m\times n\times m\times n$, its largest M-eigenvalue is upper bounded by $2\sqrt m + 2\sqrt n$. For a piezoelectric-type tensor of size $n\times n\times n$, its largest C-eigenvalue is upper bounded by $3\sqrt n$.  The definitions of  the aforementioned singular values and eigenvalues will be introduced later.  
Here for the $\ell^2$-singular value cases, the obtained bound is of the same order as that in Tomioka and Suzuki \cite{tomioka2014spectral}   (\!\!\cite{tomioka2014spectral} considered sub-Gaussian tensors which is wider than the current setting) and Nguyen et al. \cite{nguyen2015tensor}, provided that the order of the tensor is fixed. 
 \section{Main  Results}
 
Let $\T$ denote the space of $n_1\times \cdots \times n_d$ size real tensors. For $\mathcal A\in\T$, let $$\rho(\mathcal A):=\max_{\bignorm{\mathbf u_j}=1,\mathbf u_j\in\mathbb R^{n_j}}\innerprod{\mathcal A}{\bigotimes^d_{j=1}\nolimits\mathbf u_j}$$
denote the 
     largest  ($\ell^2$-) singular value of a tensor \cite{lim2005singular}, where $\otimes$ represents the outer product. It is equal to the spectral norm of $\mathcal A$.   Let $\bignorm{\cdot}_d$ denote the $\ell^d$-norm. Let
     \[
     \rho_{\ell^d}(\mathcal A):= \max_{\bignorm{\mathbf u_j}_d=1,\mathbf u_j\in\mathbb R^{n_j}}\innerprod{\mathcal A}{\bigotimes^d_{j=1}\nolimits\mathbf u_j }.
     \]
     Then $\rho_{\ell^d}(\mathcal A)$ is the largest $\ell^d$-singular value of $\mathcal A$ \cite{lim2005singular}. 
       	The following   can be seen as an extension of  part of the Gordon's theorem (see, e.g., \cite[Theorem 5.32]{vershynin2012introduction})   to higher-order tensors.
 \begin{theorem}
 	\label{th:random_tensor_nonsymmetric}
 	Let $\mathcal A \in \T$ whose entries are independent standard normal random variables. Then
 	\begin{align*}
 	\E\rho(\mathcal A) &\leq \sqrt {n_1} + \cdots + \sqrt{ n_d};\\
 	\E\rho_{\ell^d}(\mathcal A) &\leq 2^{\frac{d}{2}} \bigxiaokuohao{ \pi^{-\frac{1}{2}} \Gamma\left( \frac{1}{2(d-1)} +1\right)  }^{\frac{d-1}{d}}\prod^d_{j=1}n_j^{\frac{d-2}{2d}}\sum^d_{j=1}n_j^{\frac{1}{2}} \leq 2^{\frac{d-1}{2} }\prod^d_{j=1}n_j^{\frac{d-2}{2d}}\sum^d_{j=1}n_j^{\frac{1}{2}},
 	\end{align*}
  where $\Gamma(\cdot)$ is the Gamma function. 
 \end{theorem}
Concentration inequalities of $\rho(\mathcal A)$ and $\rho_{\ell^d}(\mathcal A)$ are then given in Corollaries \ref{col:concentration_spectral_norm} and \ref{col:concentration_d_spectral_norm}. 
 
 A tensor is called symmetric if its entries are invariant under any permutation of the indices. The set of $d$-th order $n$-dimensional symmetric tensors is denoted as $\SSS$. The largest Z- ($\ell^2$-) eigenvalue of $\mathcal A \in \SSS$ is given by \cite{qi2005eigenvalues,lim2005singular}
 \[
 \rho_z(\mathcal A) := \max_{\bignorm{\mathbf u}=1,\mathbf u\in \mathbb R^n}\innerprod{\mathcal A}{\overbrace{\mathbf u\otimes\cdots\otimes \mathbf u}^{d {~\rm times}}}.
 \]
 The largest H- ($\ell^k$-) eigenvalue of $\mathcal A\in\SSS$ is given by \cite{qi2005eigenvalues,lim2005singular}\footnote{When $d$ is even, the definitions of \cite{qi2005eigenvalues,lim2005singular} coincide. Here we mainly follow the definition of \cite{lim2005singular}.}
 \[
 \rho_h(\mathcal A) := \max_{\bignorm{\mathbf u}_d=1,\mathbf u\in \mathbb R^n}\innerprod{\mathcal A}{\overbrace{\mathbf u\otimes\cdots\otimes \mathbf u}^{d {~\rm times}}},
 \]
 where $\bignorm{\cdot}_d$ denotes the $\ell^d$-norm.
 \begin{theorem}
 	\label{th:random_tensor_symmetric}
 	Let $\mathcal A \in  \mathbb S^{n^d} $ be  a Gaussian random tensor. Then
 	\begin{align*}
 	\E \rho_z(\mathcal A) & \leq d\sqrt n;\nonumber\\
 	\E \rho_h(\mathcal A) &\leq d\cdot 2^{\frac{d}{2}} \bigxiaokuohao{\pi^{-\frac{1}{2}} \Gamma\left( \frac{1}{2(d-1)}+1  \right) }^{\frac{d-1}{d}}   n^{\frac{d-1}{2}}   \leq d\cdot 2^{\frac{d-1}{2}} n^{\frac{d-1}{2}}  .\label{eq:th:sym:1}
 	\end{align*}
 \end{theorem}

A fourth-order tensor $\mathcal A\in \mathbb R^{m\times n\times m\times n}$ is called partially symmetric  if $\mathcal A_{ijkl}=\mathcal A_{kjil}=\mathcal A_{ilkj}=\mathcal A_{klij}$. Such a set is denoted as $\M$. The largest M-eigenvalue is given by  \cite{qi2009conditions}
\[
\rho_m(\mathcal A) := \max_{\bignorm{\mathbf u}=1,\mathbf u\in\mathbb R^m,\bignorm{\mathbf v}=1,\mathbf v\in\mathbb R^n}\innerprod{\mathcal A}{\mathbf u\otimes\mathbf v\otimes\mathbf u\otimes\mathbf v}.
\]
 \begin{theorem}
	\label{th:random_tensor__part_symmetric}
	Let $\mathcal A \in  \M $ be  a Gaussian random tensor. Then
	\[
	\E \rho_m(\mathcal A) \leq 2\sqrt m + 2\sqrt n.
	\]
\end{theorem} 

Let $\mathcal A\in\mathbb R^{n\times n\times n}$ satisfy $\mathcal A_{ijk}=\mathcal A_{ikj}$; $\mathcal A$ then is called a piezoelectric-type tensor. Such a set is denoted as $\C$. The largest C-eigenvalue is given by \cite{chen2017spectral}
\[
\rho_c(\mathcal A) := \max_{\bignorm{\mathbf u}=\bignorm{\mathbf v}=1,\mathbf u,\mathbf v\in\mathbb R^n}\innerprod{\mathcal A}{\mathbf u\otimes\mathbf v\otimes\mathbf v}.
\]
\begin{theorem}
	\label{th:random_tensor__piezoelectric}
	Let $\mathcal A \in  \C $ be  a Gaussian random tensor. Then
	\[
	\E \rho_c(\mathcal A) \leq 3\sqrt n.
	\]
\end{theorem} 

The definitions of (partially) symmetric Gaussian tensors will be given later.

 \subsection{Proofs}
 We first prove Theorem \ref{th:random_tensor_nonsymmetric}; the following lemmas are needed. 
 \begin{lemma}
 	\label{lem:1}
 	Given $\mathbf u_{j},\mathbf u^\prime_{j}\in \mathbb R^{n_j}$, $\bignorm{\mathbf u_j}=\bignorm{\mathbf u^\prime_j}=1$, $j=1,\ldots,d$, $d\geq 2$,  there holds
 	\[
 	\sum^{n_1}_{i_1=1}\cdots\sum^{n_d}_{i_d=1}\bigxiaokuohao{ \mathbf u_{1,i_1}\cdots\mathbf u_{d,i_d} - \mathbf u^\prime_{1,i_1}\cdots\mathbf u^\prime_{d,i_d}  }^2 \leq \sum^d_{j=1}\bignorm{\mathbf u_j-\mathbf u^\prime_j}^2,
 	\]
 	where $\mathbf u_{j,i}$ denotes the $i$-th entry of $\mathbf u_j $. 
 \end{lemma}
\begin{proof}
We use an induction method to show the results. When $d=2$, the required inequality follows from the following relation
\begin{eqnarray}
&&\bignorm{\mathbf u_1-\mathbf u_1^\prime}^2 + \bignorm{\mathbf u_2-\mathbf u_2^\prime}^2 - \sum^{n_1}_{i_1=1}\sum^{n_2}_{i_2=1}\bigxiaokuohao{\mathbf u_{1,i_1}\mathbf u_{2,i_2} - \mathbf u_{1,i_2}^\prime\mathbf u_{2,i_2}^\prime   }^2 \nonumber\\
&=& 2\bigxiaokuohao{ 1- \innerprod{\mathbf u_1}{\mathbf u_1^\prime} - \innerprod{\mathbf u_2}{\mathbf u_2^\prime} + \innerprod{\mathbf u_1}{\mathbf u_1^\prime}\innerprod{\mathbf u_2}{\mathbf u_2^\prime}   } \geq 0.
\end{eqnarray}
Assume that the inequality holds when $d=m\geq 2$. When $d=m+1$, we denote $\mathbf v:= \bigotimes^m_{j=1}\mathbf u_j$; correspondingly, $\mathbf v^\prime:= \bigotimes^m_{j=1}\mathbf u_j^\prime$. There holds
\begin{eqnarray}
&&	\sum^{n_1}_{i_1}\cdots\sum^{n_d}_{i_d=1}\bigxiaokuohao{ \mathbf u_{1,i_1}\cdots\mathbf u_{d,i_d} - \mathbf u^\prime_{1,i_1}\cdots\mathbf u^\prime_{d,i_d}  }^2 \nonumber\\
	 &=& \sum^{n_1\cdots n_m}_{i=1}\sum^{n_{m+1}}_{i_{m+1}=1}\bigxiaokuohao{\mathbf v_i\mathbf u_{m+1,i_{m+1}} - \mathbf v_i^\prime\mathbf u_{m+1,i_{m+1}}^\prime   }^2 \nonumber\\
	&\leq& \bignorm{\mathbf v-\mathbf v^\prime}^2 + \bignorm{\mathbf u_{m+1}-\mathbf u^\prime_{m+1}}^2 \label{eq:proof:1}\\
	&=& \sum^{n_1}_{i_1=1}\cdots\sum^{n_m}_{i_m=1}\bigxiaokuohao{ \mathbf u_{1,i_1}\cdots\mathbf u_{m,i_m} -\mathbf u_{1,i_1}^\prime\cdots\mathbf u_{m,i_m}^\prime  }^2 + \bignorm{\mathbf u_{m+1}-\mathbf u^\prime_{m+1}}^2 \label{eq:proof:2} \\
	&\leq& \bignorm{\mathbf u_1 -\mathbf u_1^\prime}^2 + \cdots + \bignorm{\mathbf u_{m+1}-\mathbf u_{m+1}^\prime}^2, \label{eq:proof:3}
\end{eqnarray}
where \eqref{eq:proof:1} follows from the $d=2$ case, \eqref{eq:proof:2} is due to the definition of $\mathbf v$ and $\mathbf v^\prime$, and \eqref{eq:proof:3} is due to the assumption that $d=m$ holds. Thus induction method tells us that the inequality in question holds for all $d\geq 2$. This completes the proof.
\end{proof}

\begin{lemma}
	\label{lem:1_H}
	Given $\mathbf u_{j},\mathbf u^\prime_{j}\in \mathbb R^{n_j}$, $\bignorm{\mathbf u_j}_k=\bignorm{\mathbf u^\prime_j}_k=1$, $j=1,\ldots,d$, $d\geq 2$,  and $k\geq 3$ is an integer,  there holds
		\[
		\sum^{n}_{i_1,\ldots,i_d=1} \bigxiaokuohao{ \mathbf u_{i_1}\cdots\mathbf u_{i_d} - \mathbf u^\prime_{i_1}\cdots\mathbf u^\prime_{i_d}  }^2 \leq 2^{d-1} \bigxiaokuohao{\sum^d_{j=1} \prod_{i\neq j}^{d} n_i^{\frac{k-2}{k}} \bignorm{\mathbf u_j-\mathbf u_j^{\prime}}^2 }.
		\]
\end{lemma}
\begin{proof}
	The result that $\max_{\bignorm{\mathbf u_j}_k=1 }\bignorm{\mathbf u_j}^2 = n_j^{\frac{k-2}{k}}$ will be used. 
	
		When $d=2$, 
	\begin{align*}
		\sum^{n}_{i_1,i_2=1}\bigxiaokuohao{\mathbf u_{1,i_1}\mathbf u_{2,i_2} - \mathbf u_{1,i_2}^\prime\mathbf u_{2,i_2}^\prime   }^2 &\leq 2 \bignorm{\mathbf u_2}^2\bignorm{\mathbf u_1-\mathbf u^{\prime}_1} + 2\bignorm{\mathbf u^{\prime}_1}^2\bignorm{\mathbf u_2-\mathbf u^{\prime}_2}^2\\
		&\leq 2 n_2^{\frac{k-2}{k}} \bignorm{\mathbf u_1-\mathbf u^{\prime}_1}^2 + 2n_1^{\frac{k-2}{k}}\bignorm{\mathbf u_2-\mathbf u_2^{\prime}}^2 .
	\end{align*}
	Assume that when $d=m$ the assertion holds. Denote $\mathbf v,\mathbf v^{\prime}$ similar to those in Lemma \ref{lem:1}. When $d=m+1$,
	\begin{align*}
	&	\sum^{n_1}_{i_1=1}\cdots\sum^{n_d}_{i_j=1} \bigxiaokuohao{ \mathbf u_{1,i_1}\cdots\mathbf u_{d,i_d} - \mathbf u^\prime_{1,i_1}\cdots\mathbf u^\prime_{d,i_d}  }^2 \\
		\leq &~2\bignorm{\mathbf v}^2\bignorm{\mathbf u_{m+1}-\mathbf u^{\prime}_{m+1}}^2 + 2\bignorm{\mathbf u^{\prime}_{m+1}}^2\bignorm{\mathbf v-\mathbf v^{\prime}}^2\\
		\leq &~2 \prod^m_{j=1}n_j^{\frac{k-2}{k}}  \bignorm{\mathbf u_{m+1} -\mathbf u^{\prime}_{m+1} }^2 + 2 n_{m+1}^{\frac{k-2}{k}}\cdot 2^{m-1} \bigxiaokuohao{ \sum^m_{j=1}\prod_{i\neq j}^{m} n_i^{\frac{k-2}{k}}\bignorm{\mathbf u_j-\mathbf u_j^{\prime}}^2  }  \\
		\leq &~ 2^m  \bigxiaokuohao{\sum^{m+1}_{j=1}\prod^{m+1}_{i\neq j}n_i^{\frac{k-2}{k}}\bignorm{\mathbf u_j-\mathbf u_j^{\prime}}^2  } ,
	\end{align*}
	in which $\max_{\bignorm{\mathbf u_1}_k=\cdots=\bignorm{\mathbf u_m}_k=1 }\bignorm{\mathbf v}^2 = \max_{\bignorm{\mathbf u_1}_k=\cdots=\bignorm{\mathbf u_m}_k=1 } \prod^m_{j=1}\bignorm{\mathbf u_j}^2 = \prod^m_{j=1}n_j^{\frac{k-2}{k}}$. The result follows. 
\end{proof}

 The proof is also relied on the Slepian’s inequality for Gaussian processes, which is stated in the following lemma. Note that a Gaussian process $(X_t)_{t\in T}$ is a collection of centered normal random variables $X_t$ on the same probability space, indexed by points t in an abstract set $ T$.
\begin{lemma}(\cite[Sect. 3.3]{ledoux2013probability}) \label{lem:slepian}
   Consider two Gaussian processes $\left(X_{t}\right)_{t \in T}$         and  $\left(Y_{t}\right)_{t \in T} $   
          whose increments satisfy the inequality  
         $ \mathbb{E}\left|X_{s}-X_{t}\right|^{2} \leq \mathbb{E} \left| Y_{s}-    Y_{t}\right|^{2} $
            for all  $s, t \in T$.    Then  
            \[
            \mathbb{E} \sup _{t \in T} X_{t} \leq \mathbb{E} \sup _{t \in T} Y_{t}.
            \]
\end{lemma} 
 
 With Lemmas \ref{lem:1} and \ref{lem:slepian}, we can prove Theorem \ref{th:random_tensor_nonsymmetric}.
 \begin{proof}[Proof of Theorem \ref{th:random_tensor_nonsymmetric}]
Denote 
\[
T:= \{\mathbf u_1\in\mathbb R^{n_1}\mid\bignorm{\mathbf u_1}=1  \} \times \cdots\times \{\mathbf u_d\in\mathbb R^{n_d}\mid \bignorm{\mathbf u_d}=1   \}.
\]
We also define
\begin{equation}\label{def:Xu}
X_{\mathbf u_j}:=\innerprod{\mathcal A}{\bigotimes^d_{j=1}\nolimits\mathbf u_j} 
\end{equation}
and
\[
Y_{\mathbf u_j}:= \innerprod{\mathbf h_1}{\mathbf u_1} + \cdots+\innerprod{\mathbf h_d}{\mathbf u_d},
\]
where $\mathbf h_j\in\mathbb R^{n_j}$, $j=1,\ldots,d$ are independent standard Gaussian random vectors. 
It is clear that $(X_{\mathbf u_j})_{(\mathbf u_1,\ldots,\mathbf u_d)\in T  }$ and 
$(Y_{\mathbf u_j})_{(\mathbf u_1,\ldots,\mathbf u_d)\in T }$ are Gaussian process. It also holds that
\[
\mathbb E X_{\mathbf u_j} = 0,~{\rm and}~\E  Y_{\mathbf u_j}=0.
\]
Moreover, given $(\mathbf u_1,\ldots,\mathbf u_d)\in T,(\mathbf u_1^\prime,\ldots,\mathbf u^\prime_d)\in T$, 
we have
\begin{eqnarray}
\E\bigjueduizhi{ X_{\mathbf u_j} - X_{\mathbf u^\prime_{j}}  }^2 &=& \E \innerprod{\mathcal A}{\bigotimes^d_{j=1}\mathbf u_j}^2 - 2 \E\innerprod{\mathcal A}{\bigotimes^d_{j=1}\mathbf u_j}\innerprod{\mathcal A}{\bigotimes^d_{j=1}\mathbf u_j^{\prime}} + \E\innerprod{\mathcal A}{\bigotimes^d_{j=1}\mathbf u_j^{\prime}}^2 \nonumber\\
&=& \prod^d_{j=1}\innerprod{\mathbf u_j}{\mathbf u_j}  - 2\prod^d_{j=1}\innerprod{\mathbf u_{j}}{\mathbf u^\prime_j} + \prod^d_{j=1}\innerprod{\mathbf u_j^\prime}{\mathbf u_j^\prime} \nonumber\\
&=&  	\sum^{n_1}_{i_1}\cdots\sum^{n_d}_{i_d=1}\bigxiaokuohao{ \mathbf u_{1,i_1}\cdots\mathbf u_{d,i_d} - \mathbf u^\prime_{1,i_1}\cdots\mathbf u^\prime_{d,i_d}  }^2, \label{eq:proof:4}
\end{eqnarray}
where the second equality follows from   that all the entries of $\mathcal A$ are independent standard normal variables.  On the other hand, it follows from the definition of $Y_{\mathbf u_j}$ that
\begin{eqnarray} \label{eq:proof:5}
\E \bigjueduizhi{Y_{\mathbf u_j} - Y_{\mathbf u^\prime_j}}^2 &=& \sum^d_{j=1}\bigxiaokuohao{\bignorm{\mathbf u_j}^2 - 2\innerprod{\mathbf u_j}{\mathbf u_j^\prime} + \bignorm{\mathbf u^\prime_j}^2}\nonumber\\
&=& \sum^d_{j=1}\bignorm{\mathbf u_j - \mathbf u^\prime_j}^2.
\end{eqnarray}
  \eqref{eq:proof:4} and \eqref{eq:proof:5} in connection with Lemma \ref{lem:1} yields
  \[
  \E\bigjueduizhi{ X_{\mathbf u_j} - X_{\mathbf u^\prime_{j}}  }^2  \leq \E \bigjueduizhi{Y_{\mathbf u_j} - Y_{\mathbf u^\prime_j}}^2,
  \]
  which together with Lemma \ref{lem:slepian} shows that
  \begin{eqnarray*}
  \E \rho(\mathcal A)	&=& \E\sup_{ (\mathbf u_1,\ldots,\mathbf u_d)\in T    }X_{\mathbf u_j} \leq \E \sup_{ (\mathbf u_1,\ldots,\mathbf u_d)\in T   }Y_{\mathbf u_j} \nonumber\\
  &\leq& \E\bignorm{\mathbf h_1} + \cdots + \E \bignorm{\mathbf h_d}\nonumber\\
  &\leq& \sqrt n_1 + \cdots + \sqrt n_d,
  \end{eqnarray*}
  where the last inequality follows from Jensen's inequality.

  To estimate $\rho_{\ell^d}(\mathcal A)$, we keep $X_{\mathbf u_j}$ as that in \eqref{def:Xu}, while   redefine $Y_{\mathbf u_j}$ as
  \[
Y_{\mathbf u_j}:= \alpha_1  \innerprod{\mathbf h_1}{\mathbf u_1} +   \cdots + \alpha_d\innerprod{\mathbf h_d}{\mathbf u_d},~{\rm with}~ \alpha_j:= 2^{\frac{d-1}{2}}\prod^d_{i\neq j}n_i^{\frac{d-2}{2d}}.
  \]
  Then $\E Y_{\mathbf u_j}=0$.  Correspondingly, we redefine
  \[
  T:= \{\mathbf u_1\in\mathbb R^{n_1}\mid\bignorm{\mathbf u_1}_d=1  \} \times \cdots\times \{\mathbf u_d\in\mathbb R^{n_d}\mid \bignorm{\mathbf u_d}_d=1   \}.
  \]
  
  For $(\mathbf u_1,\ldots,\mathbf u_d)\in T$, it is clear that \eqref{eq:proof:4} still holds for $\E X_{\mathbf u_j}$, while 
  \[
  \E \bigjueduizhi{ Y_{\mathbf u_j}-Y_{\mathbf u_j}^{\prime}}^2 = 2^{d-1}\bigxiaokuohao{ \sum^d_{j=1} \prod^d_{i\neq j}n_i^{\frac{d-2}{d}} \bignorm{\mathbf u_j-\mathbf u_j^{\prime}}^2   }. 
  \]
  Thus it follows from  Lemma \ref{lem:1_H}  that we still have
  $
    \E\bigjueduizhi{ X_{\mathbf u_j} - X_{\mathbf u^\prime_{j}}  }^2  \leq \E \bigjueduizhi{Y_{\mathbf u_j} - Y_{\mathbf u^\prime_j}}^2.
  $
  
  To estimate $\rho_{\ell^d}(\mathcal A)$, it suffices to compute $\E\sup_{\bignorm{\mathbf u_j}_d=1}\innerprod{\mathbf h_j}{\mathbf u_j}$. Define $p:=\frac{d}{d-1}$. From the definition of the dual norm, we have $\E \sup_{\bignorm{\mathbf u_j}_d=1}\innerprod{\mathbf h_j}{\mathbf u_j} = \E \bignorm{\mathbf h_j}_p$. Since $p>1$, by Yensen's inequality, $\bigxiaokuohao{\E \bignorm{\mathbf h_j}_p}^p \leq \E \bignorm{\mathbf h_j}^p_p = \E \sum^n_{i_j=1} |\mathbf h_{j,i_j }|^p$, while  for any $i_j$, 
  \begin{align*}
  	\E |\mathbf h_{j,i_j}|^p &= \frac{1}{\sqrt{2\pi}}\int_{+\infty}^{-\infty} |x|^p \exp\bigxiaokuohao{ -\frac{x^2}{2}  } dx\\ 
  	&= \frac{2}{\sqrt{2\pi}}\int_{0}^{+\infty} x^p \exp\bigxiaokuohao{ -\frac{x^2}{2}  } dx\\
  	&= 2^{ \frac{p-1}{2} }\cdot \frac{2}{\sqrt{2\pi}} \int_{0}^{+\infty} \bigxiaokuohao{ y }^{\frac{p-1}{2}+1-1} \exp\bigxiaokuohao{ -y  } d y\\
  	&=    {\frac{2^{ \frac{d}{2(d-1)} }}{\sqrt \pi}}\Gamma\bigxiaokuohao{\frac{1}{2(d-1)}+1}.
  \end{align*}
Therefore, 
  \begin{align*}
  	\rho_{\ell^d}(\mathcal A)	&=\E\sup_{ (\mathbf u_1,\ldots,\mathbf u_d)\in T}X_{\mathbf u_j} \leq  \E \sup_{(\mathbf u_1,\ldots,\mathbf u_d)\in T} Y_{\mathbf u_j} \\
  	=&  \alpha_1\E  \sup_{\bignorm{\mathbf u_1}_d =1} \innerprod{\mathbf h_1}{\mathbf u_1} + \cdots + \alpha_d \sup_{ \bignorm{\mathbf u_d}_d =1}\innerprod{\mathbf h_d}{\mathbf u_d}\\
  	  =& \alpha_1\E\bignorm{\mathbf h_1}_p + \cdots + \alpha_d\E\bignorm{\mathbf h_d}_p \\
  	\leq &\alpha_1 \bigxiaokuohao{\E \sum^{n_1}_{i_1=1}|\mathbf h_{1,i_1}|^p  }^{\frac{1}{p}} + \cdots + \alpha_d  \bigxiaokuohao{\E \sum^{n_d}_{i_d=1}|\mathbf h_{d,i_d}|^p  }^{\frac{1}{p}}   \\
  	= &\sum^d_{j=1}2^{\frac{d-1}{2}}\prod^d_{i\neq j}n_i^{\frac{d-2}{2d}} \cdot\bigxiaokuohao{  {\frac{2^{ \frac{d}{2(d-1)} }}{\sqrt \pi}}\Gamma\bigxiaokuohao{\frac{1}{2(d-1)}+1} \cdot n_j   }^{\frac{d-1}{d}} \\
  	=& 2^{\frac{d}{2}} \bigxiaokuohao{ \pi^{-\frac{1}{2}} \Gamma( \frac{1}{2(d-1)} +1)  }^{\frac{d-1}{d}}\prod^d_{j=1}n_j^{\frac{d-2}{2d}}\sum^d_{j=1}n_j^{\frac{1}{2}}   . 
  \end{align*}

To see the last inequality of Theorem \ref{th:random_tensor_nonsymmetric}, since $p<2$, by Yensen's inequality, $\bigxiaokuohao{\E |\mathbf h_{j,i_j}|^p}^{1/p}\leq \bigxiaokuohao{\E |{\mathbf h_{j,i_j}}|^2}^{1/2}$, and so 
\[
\bigxiaokuohao{\E \sum^{n_j}_{i_j=1}\nolimits|\mathbf h_{j,i_j}|^p  }^{\frac{1}{p}} \leq n_j^{\frac{1}{p}-\frac{1}{2}}\bigxiaokuohao{\E \sum^{n_j}_{i_j=1}\nolimits|\mathbf h_{j,i_j}|^2  }^{\frac{1}{2}}=n_j^{\frac{d-1}{d}}.
\]
  This together with the definition of $\alpha_j$ gives the last inequality. 
  The proof has been completed.
 \end{proof}

We then present the concentration inequality for $\rho({\mathcal A})$. The following propositions are useful. 
 \begin{proposition}
 	\label{lem:1-lip}
 	The largest singular value function $\rho({\cdot}):\T \rightarrow \mathbb R$ is   Lipschitz continuous   with Lipschitz constant $1$.
 \end{proposition}
\begin{proof}
	For any $\mathcal A,\mathcal B\in\T$, assume that $(\mathbf y_1,\ldots,\mathbf y_d)\in T$ is such that   $\innerprod{\mathcal B}{\bigotimes^d_{j=1}\mathbf y_j}=\rho({\mathcal B})$. Thus
	\[
	\rho({\mathcal A})\geq   \innerprod{\mathcal A}{\bigotimes^d_{j=1}\mathbf y_j} = \rho({\mathcal B}) + \innerprod{\mathcal A-\mathcal B}{\bigotimes^d_{j=1}\mathbf y_j} \geq \rho({\mathcal B}) - \rho({\mathcal A-\mathcal B}).
	\]
On the other hand, since $\rho(\cdot)=\bignorm{\cdot}_2$, the tensor spectral norm,	  $\rho({\mathcal A}) \leq \rho({B})  + \rho({\mathcal A-\mathcal B}) $.   As a result,
	\[
	\bigjueduizhi{ \rho({\mathcal A})  - \rho({\mathcal B})  } \leq \rho({ \mathcal A - \mathcal B})  =\bignorm{\mathcal A-\mathcal B}_2 \leq \bignorm{\mathcal A-\mathcal B}_F, 
	\]
	where the last inequality follows from that the tensor spectral norm is smaller than its Frobenius norm. This completes the proof.
\end{proof}

\begin{proposition}[c.f. \cite{ledoux2001concentration}]
	\label{prop:concen_1-lip}
	Let $F(\cdot)$ be a real-valued $1$-Lipschitz function, and let $\mathbf x$ be a standard normal random vector. Then for any $t>0$, there holds 
	\[
  \Prob{F(\mathbf x) - \E F(\mathbf x)  > t }  \leq \exp( -t^2/2). 
	\]
\end{proposition}

With the above propositions and Theorem \ref{th:random_tensor_nonsymmetric}, we have the following one-side concentration inequality for the   largest singular value concerning a gaussian tensor.
\begin{corollary}
	\label{col:concentration_spectral_norm}
	Under the setting of Theorem \ref{th:random_tensor_nonsymmetric}, for any $t>0$, we have
	\[
  \Prob{ \rho({\mathcal A}) >\sqrt n_1 + \cdots + \sqrt n_d  + t }  \leq \exp( -t^2/2). 	 
	\]
\end{corollary}
Similarly we have
\begin{corollary}
	\label{col:concentration_d_spectral_norm}
	Under the setting of Theorem \ref{th:random_tensor_nonsymmetric}, for any $t>0$, we have
	\[
	\Prob{ \rho_{\ell^d}({\mathcal A}) >2^{\frac{d-1}{2} }\prod^d_{j=1}\nolimits n_j^{\frac{d-2}{2d}}\sum^d_{j=1}\nolimits n_j^{\frac{1}{2}}  + t }  \leq \exp( -t^2/2). 	 
	\]
\end{corollary}

We then consider the  symmetric cases. For an indices tuple $(i_1,\ldots,i_d)$, denote $\pi(i_1,\ldots,i_d)$ a permutation of $(i_1,\ldots,i_d)$; denote $\Pi(i_1,\ldots,i_d)$ the set of all the possible permutations of $(i_1,\ldots,i_d)$, and ${\rm card}(\cdot)$ the cardinality of a set. 

A symmetric Gaussian tensor $\mathcal A\in\SSS$ is defined as follows: 
\begin{definition}
	\label{def:gaussian_sym}
	Let $\mathcal A\in\SSS$ be defined as that the $\binom{n+d-1}{d}$ entries $\mathcal A_{i_1,\ldots,i_d}$  are independently drawn from  the Gaussian distribution, in which $i_1\leq i_2\leq\cdots \leq i_d$, $1\leq i_j\leq n$, $1\leq j\leq d$, satisfying
	\[
	A_{i_1\cdots i_d} \sim N\left(0,\frac{d}{ {\rm card}(\Pi(i_1,\ldots,i_d)  ) }\right).
	\]
Then,  set $\mathcal A_{\pi(i_1,\ldots,i_d )} = \mathcal A_{i_1,\ldots,i_d}$ for any permutation $\tau(i_1,\ldots,i_d)\in \Pi(i_1,\ldots,i_d)$. We call $\mathcal A$ a symmetric Gaussian random tensor. 
\end{definition}

For example, for $\mathcal A\in\mathbb S^{3^3}$,  let $\mathcal A_{111} \sim N(0,3), \mathcal A_{211}=\mathcal A_{121}=\mathcal A_{112}\sim N(0,1)$,  and $\mathcal A_{321}=\mathcal A_{312} =\mathcal A_{213}=\mathcal A_{231} = \mathcal A_{132} = \mathcal A_{123}\sim N(0,\frac{1}{2})$. Other entries are drawn from the same principle.

\begin{remark}
Definition \ref{def:gaussian_sym} is a higher-order generalization of the Gaussian Orthogonal Ensemble  (GOE) for symmetric random tensors: it is clear that	when $d=2$, according to Definition \ref{def:gaussian_sym}, $\mathcal A_{ji}=\mathcal A_{ij}\sim N(0,1),i\neq j$, and $\mathcal A_{ii}\sim N(0,2)$, which is exactly the definition of GOE.  
\end{remark}

\begin{proof}[Proof of Theorem \ref{th:random_tensor_symmetric}]
Similar to the proof of Theorem \ref{th:random_tensor_nonsymmetric}, let
$$
X_{\mathbf u}:= \innerprod{\mathcal A}{\overbrace{\mathbf u\otimes\cdots\otimes \mathbf u}^{d {~\rm times}}},~{\rm and}~ Y_{\mathbf u}:= d\innerprod{\mathbf h}{\mathbf u},
$$
where $\mathbf h$ is a standard Gaussian random vector.  Then $\E X_{\mathbf u}=\E Y_{\mathbf u}=0$.  
According to Definition \ref{def:gaussian_sym}, after some computations, it can be verified that
\begin{eqnarray}
\E \bigjueduizhi{ X_{\mathbf u} - X_{\mathbf u^{\prime}} }^2 &=& d\cdot \bigxiaokuohao{\innerprod{\mathbf u}{\mathbf u}^d- 2\innerprod{\mathbf u}{\mathbf u^{\prime}}^d+ \innerprod{\mathbf u^{\prime}}{\mathbf u^{\prime}}^d}\nonumber\\
&=& 
 d \cdot \sum^n_{i_1,\ldots,i_d=1}\bigxiaokuohao{\mathbf u_{i_1}\cdots \mathbf u_{i_d} - \mathbf u^{\prime}_{i_1}\cdots \mathbf u^{\prime}_{i_d}   }^2\nonumber\\
&\leq& d^2\bignorm{\mathbf u-\mathbf u^{\prime}}^2 = \E\bigjueduizhi{Y_{\mathbf u} - Y_{\mathbf u^{\prime}} }^2. \label{eq:proof:6}
\end{eqnarray}
Then $$\E \rho_z(\mathcal A) = \E \sup_{\bignorm{\mathbf u}=1 } X_{\mathbf u} \leq \E\sup_{\bignorm{\mathbf u}=1} Y_{\mathbf u}\leq d\sqrt n.$$ 

To estimate $\E\rho_h(\mathcal A)$, we redefine $Y_{\mathbf u}$ as 
\[
Y_{\mathbf u}:=  \alpha\innerprod{\mathbf h}{\mathbf u},~{\rm with}~ \alpha:= d\cdot 2^{\frac{d-1}{2}}n^{\frac{(d-1)(d-2)}{2d}}.
\]
According to Lemma \ref{lem:1_H} and similar to \eqref{eq:proof:6},  $\E \bigjueduizhi{ X_{\mathbf u} - X_{\mathbf u^{\prime}} }^2\leq \E\bigjueduizhi{Y_{\mathbf u} - Y_{\mathbf u^{\prime}} }^2$    when $\bignorm{\mathbf u}_d=1$. Similar to the estimate of $\rho_{\ell^d}(\mathcal A)$ in the proof of Theorem \ref{th:random_tensor_nonsymmetric},
\begin{align*}
\rho_h(\mathcal A)	&=\E\sup_{\bignorm{\mathbf u}_d=1}X_{\mathbf u} \leq  \E \sup_{\bignorm{\mathbf u}_d=1} Y_{\mathbf u} =  \alpha\E  \sup_{\bignorm{\mathbf u}_d =1} \innerprod{\mathbf h}{\mathbf u}= \alpha\E\bignorm{\mathbf h}_p \\
&\leq \alpha \bigxiaokuohao{\E \sum^n_{i=1}|\mathbf h_i|^p  }^{\frac{1}{p}} = d\cdot 2^{\frac{d}{2}} \bigxiaokuohao{\pi^{-\frac{1}{2}} \Gamma( \frac{1}{2(d-1)}+1  ) }^{\frac{d-1}{d}}   n^{\frac{d-1}{2}} \\
 & \leq d\cdot 2^{\frac{d-1}{2}} n^{\frac{d-1}{2}} . 
\end{align*}
\end{proof}

The partially symmetric Gaussian tensor $\mathcal A\in \M$ is defined as follows:
\begin{definition}
	\label{def:gaussian_part_sym}
Let $\mathcal A\in\SSS$ be defined as that the   are independently drawn from  the Gaussian distribution, in which $i_1\leq k$, $j\leq l$, satisfying
\[
A_{ijkl} \sim N\left(0,\frac{2}{ {\rm card}(\Pi(i,j,k,l)  ) }\right);
\]
 in this context, $\Pi(i,j,k,l)$ means all the possible permutations of $(i,j,k,l)$ in the partial symmetry sense. 
Then,  set $\mathcal A_{\pi(i_1,\ldots,i_d )} = \mathcal A_{i_1,\ldots,i_d}$ for any permutation $\tau(i_1,\ldots,i_d)\in \Pi(i_1,\ldots,i_d)$. We call $\mathcal A$ a  partially symmetric Gaussian random tensor. 
\end{definition}

For example, for $\mathcal A\in\mathbb S_M^{3\times 3\times 3\times 3}$, let $\mathcal A_{1111}\sim N(0,2)$, $\mathcal A_{1212}\sim N(0,2)$, $\mathcal A_{3212}=\mathcal A_{1232}\sim N(0,1)$, and $\mathcal A_{2312}=\mathcal A_{2213}=\mathcal A_{1322}=\mathcal A_{1223} \sim N(0,\frac{1}{2})$. Other entries are drawn from the same principle.

Definition \ref{def:gaussian_part_sym} is also a higher-order generalization of the matrix GOE in cases that $m=1$ or $n=1$.

\begin{proof}[Proof of Theorem \ref{th:random_tensor__part_symmetric}]
Define $X_{\mathbf u,\mathbf v}:= \innerprod{\mathcal A}{\mathbf u\otimes\mathbf v\otimes \mathbf u\otimes \mathbf v}$ and $Y_{\mathbf u,\mathbf v}:= 2\innerprod{\mathbf h_1}{\mathbf u} + 2\innerprod{\mathbf h_2}{\mathbf v}$, where $\mathbf h_1,\mathbf h_2$ are independent standard Gaussian vectors. 
\begin{eqnarray*}
\E\bigjueduizhi{ X_{\mathbf u,\mathbf v} - X_{\mathbf u^{\prime},\mathbf v^{\prime}} }^2 &=& 2\bigxiaokuohao{ 1 - \innerprod{\mathbf u}{\mathbf u^{\prime}}^2\innerprod{\mathbf v}{\mathbf v^{\prime}}^2  +1} \\
&=& 2\sum^m_{i,k=1}\sum^n_{j,l=1}\bigxiaokuohao{\mathbf u_i\mathbf v_j \mathbf u_k\mathbf v_l- \mathbf u_i^{\prime}\mathbf v_j^{\prime} \mathbf u_k^{\prime}\mathbf v_l^{\prime} }^2\\
& \leq& 4 \bignorm{\mathbf u-\mathbf u^{\prime}}^2 + 4\bignorm{\mathbf v-\mathbf v^{\prime}}^2 = \E \bigjueduizhi{ Y_{\mathbf u,\mathbf v} - Y_{\mathbf u^{\prime},\mathbf v^{\prime}} }^2.
\end{eqnarray*}
Then $\E\rho_m(\mathcal A) \leq \E \sup_{ \bignorm{\mathbf u}=\bignorm{\mathbf v}=1 }Y_{\mathbf u,\mathbf v}=2\sqrt m + 2\sqrt n$.
\end{proof}

The piezoelectric-type Gaussian tensor $\mathcal A\in \C$ is defined as follows:
\begin{definition}
	\label{def:gaussian_piezoelectric}
	Let $\mathcal A\in\C$ be defined as that the   are independently drawn from  the Gaussian distribution, in which $j\leq k$, satisfying
	\[
	A_{ijk} \sim N\left(0, 1 \right), j\neq k,~{\rm and}~ A_{ijj}\sim N(0,2).
	\]
	Then,  set $\mathcal A_{ikj} = \mathcal A_{ijk}$ for  $j<k$. We call $\mathcal A$ a piezoelectric-type Gaussian random tensor. 
\end{definition}

\begin{proof}[Proof of Theorem \ref{th:random_tensor__piezoelectric}]
The proof is the same as the previous theorems. 
\end{proof}

{\footnotesize \section*{Acknowledgment} This work was supported by the National Natural Science Foundation of China Grant 11801100 and the Fok Ying Tong Education Foundation Grant 171094.
}

   \bibliography{tensor,TensorCompletion,orth_tensor,random_tensor}
  \bibliographystyle{plain}

   \appendix

\end{document}